\documentclass[12pt, reqno]{amsart}
\usepackage{amsmath, amsthm, amscd, amsfonts, amssymb, graphicx, color}
\usepackage[usenames,dvipsnames,svgnames,table]{xcolor}
\usepackage[latin1]{inputenc}
\usepackage[spanish,english]{babel}
\usepackage{enumitem}
\newtheorem{theorem}{Theorem}[section]

\newtheorem{proposition}[theorem]{Proposition}
\newtheorem{corollary}[theorem]{Corollary}
\theoremstyle{definition}
\newtheorem{definition}[theorem]{Definition}
\newtheorem{example}[theorem]{Example}

\theoremstyle{remark}
\newtheorem{remark}[theorem]{Remark}
\numberwithin{equation}{section}

\newcommand{\R}{{\mathbb R}}

\newcommand{\N}{{\mathbb N}}
\newcommand{\XB}{{\mathbb X}}
\newcommand{\Y}{{\mathbb Y}}
\newcommand{\C}{{\mathbb C}}
\newcommand{\h}{{\mathcal H}}
\newcommand{\obj}{{\perp_{B}}}
\newcommand{\orbj}{{\perp_{B}^r}}
\newcommand{\objp}{{\perp_{B}^p}}
\newcommand{\oi}{{\perp_{I}}}
\newcommand{\oip}{{\perp_{I}^p}}

\newcommand{\oro}{{\perp_{R}}}
\newcommand{\osi}{{\perp_{SI}}}

\def\bb{\mathcal{B}}
\def\X{\mathcal{X}}

\def\cH{\mathcal{H}}

\newcommand{\bh}{\mathcal{B}(\mathcal{H})}

\newcommand{\bit}{\begin{itemize}}
\newcommand{\eit}{\end{itemize}}
\newcommand{\be}{\begin{enumerate}}
\newcommand{\ee}{\end{enumerate}}

\begin{document}

\title{A study of Orthogonality of bounded linear Operators}

\author[T. Bottazzi, C. Conde, D. Sain]{Tamara Bottazzi$^1$, Cristian Conde$^{2,3}$ and Debmalya Sain$^{4}$}

\address{$^1$ Sede Andina, Universidad Nacional de R\'io Negro, (8400) S.C. de Bariloche, Argentina}

\address{$^2$Instituto Argentino de Matem\'atica ``Alberto P. Calder\'on", Saavedra 15 3er. piso, (C1083ACA) Buenos Aires, Argentina}

\address{$^3$Instituto de Ciencias, Universidad Nacional de Gral. Sarmiento, J.	
M. Gutierrez 1150, (B1613GSX) Los Polvorines, Argentina}

\address{$^4$Indian Institute of Science, Bengaluru, Karnataka, India}

\email{tbottazzi@unrn.edu.ar}
\email{cconde@ungs.edu.ar}
\email{saindebmalya@gmail.com}

\maketitle
\begin{abstract}
	We study  Birkhoff-James orthogonality and isosceles orthogonality of bounded  linear operators between Hilbert spaces and Banach spaces.  We explore Birkhoff-James orthogonality of bounded linear operators in light of a new notion introduced by us and also discuss some of the possible applications in this regard. We also study isosceles orthogonality of bounded (positive) linear operators on a Hilbert space and  some of the related properties, including that of operators having disjoint support. We further explore the relations between Birkhoff-James orthogonality and isosceles orthogonality in a general Banach space. 
\end{abstract}

\section{Introduction and preliminaries}
The primary purpose of the present paper is to explore orthogonality of bounded linear operators between Hilbert spaces and Banach spaces. Unlike the Hilbert space case, there is no universal notion of orthogonality in a Banach space. However, it is possible to have several notions of orthogonality in such space, each of which generalizes some particular aspect of Hilbert space orthogonality. Indeed, one of the root causes of the vast differences between the geometries of Hilbert spaces and Banach spaces is the lack of a standard orthogonality notion in the later case. On the other hand, this makes the study of orthogonality of bounded linear operators an interesting and deeply rewarding area of research. Motivated by this, several authors have explored orthogonality of bounded linear operators in recent times \cite{alonso-martini-wu}, \cite{bhatta-grover}, \cite{bhatia-semrl}, \cite{benitez-fernandez-soriano}, \cite{bcmz},  \cite{chmielinski}, \cite{MSP}, \cite{PSG}, \cite{PSMK}, \cite{rao}, \cite{sain_2}, \cite{sain1}, \cite{sain_paul}, \cite{turnsek}, \cite{Pawel} and \cite{Zamani}, and have obtained many interesting results involving the geometry of operator spaces. In this paper, among other things, we extend, improve and generalize some of the earlier results on orthogonality of bounded linear operators. Without further ado, let us first establish our notations and terminologies to be used throughout the paper.

Letters $\XB, \Y$ denote Banach spaces, over the field $ \mathbb{K}\in\{\mathbb{R}, \mathbb{C}\}$.  Let $B_{\XB}=\{x\in \XB: \|x\|\leq 1\}$ and $S_{\XB}=\{x\in \XB:\|x\|=1\}$ be the unit ball and the unit sphere of $\XB$ respectively. Let $\mathcal{B}(\XB, \Y)$ and $\mathcal{K}(\XB, \Y)$ denote the Banach space of all bounded linear operators and compact operators from $\XB$ to $\Y$ respectively, endowed with the usual operator norm. We write $\mathcal{B}(\XB, \Y)=\mathcal{B}(\XB)$ and $\mathcal{K}(\XB, \Y)=\mathcal{K}(\XB)$ if $\XB=\Y$. The symbol $I_{\XB}$ stands for the identity operator on $\XB.$ We omit the suffix in case there is no confusion. We reserve the symbol $ \h $ for a Hilbert space over the field $ \mathbb{K}. $ Throughout the paper, we consider only separable Hilbert spaces. In this paper, mostly in the context of bounded linear operators, we discuss three of the most important orthogonality types in a Banach space, namely, \textit{Birkhoff-James orthogonality \cite{birkhoff} and \cite{james}, isosceles orthogonality \cite{james2} and Roberts orthogonality \cite{roberts}}.

Let us first state the relevant definitions, in the more general setting of a normed space $\X$ over $\mathbb{K}$.

\begin{definition} \label{defiBJ}
	For any two elements $x,y \in \X,$  we  say that $x$ is
	\textit{Birkhoff-James orthogonal} to  $y,$ written as $x \perp_{B} y,$ if  for all $\lambda \in \mathbb{K},$ the following holds:
	\begin{equation}
	\|x\|\leq \|x+\lambda y\|.
	\end{equation}
\end{definition}

\begin{definition} \label{defiiso}
	For any two elements $x,y \in \X$  and $\mathbb{K}=\mathbb{R},$ we  say that $x$ is
	\textit{isosceles orthogonal} to  $y,$ written as $x \perp_{I} y,$ if the following holds:
	\begin{equation}
	\|x+y\| = \|x-y\|.
	\end{equation}
	In complex normed spaces, we consider the following orthogonality relation
	\begin{equation} \label{defiiso2}
	x\perp_{I}y\Leftrightarrow \left\{
	\begin{array}{c l r l}
	\|x+y\|= \|x-y\|\\
	\|x+iy\|= \|x-iy\|. \\
	\end{array}
	\right.
	\end{equation}	
\end{definition}

\begin{definition} \label{defiRoberts}
	For any two elements $x,y \in \X,$  we  say that $x$ is
	\textit{Roberts orthogonal} to  $y,$ written as $x \perp_{R} y,$ if  for all $\lambda \in \mathbb{K},$ the following holds:
	\begin{equation}
	\|x+\lambda y\| =  \|x-\lambda y\|.
	\end{equation}
\end{definition}

It is easy to see that Roberts orthogonality implies Birkhoff-James orthogonality but the converse is not necessarily true.

In order to have a better 
description of Birkhoff-James orthogonality of bounded linear operators between Banach spaces, we introduce the following notation for any $T, A \in \bb(\XB, \Y)$:
$$
\mathcal{O}_{T, A}=\{x\in S_{\XB}: Tx\perp_{B}Ax\}.
$$
Given $T\in \bb(\XB, \Y)$, define the norm attainment set of $T$ as
$$M_{T}=\{x\in S_\XB:\ \|Tx\|=\|T\|\}.$$

As observed in \cite{bhatia-semrl}, \cite{rao}, \cite{sain_2}, \cite{sain1} and \cite{SPM}, the structure of the norm attainment set of a bounded linear operator is of central importance in studying Birkhoff-James orthogonality and smoothness of the said operator. On the other hand, it was illustrated in \cite{MSP} that the notion of the norm attainment set of a bounded linear operator is deeply related to the geometry of the space of bounded linear operators between Banach spaces. In the context of $ T \in \bh, $ the corresponding norm attainment set $ M_T $ was completely characterized in \cite{sain1}. We would like to remark that Birkhoff-James orthogonality of bounded linear operators on a finite-dimensional Hilbert space $ \h $ was completely characterized by Bhatia and $\breve{S}$emrl in \cite{bhatia-semrl}:
\[\textit{For $ T, A \in \bh, ~ T \perp_{B} A \Longleftrightarrow \mathcal{O}_{T, A} \cap M_T \neq \emptyset. $}\]
This motivates us to explore the structure of $ \mathcal{O}_{T, A}, $ for two given operators $ T, A \in \mathcal{B}(\XB, \Y). $ In order to study the properties of the set $ \mathcal{O}_{T,A}, $ in the context of a real Banach space, we require the following notions introduced in \cite{sain_2}.
\begin{definition}
	Let $ \X$ be a real  normed space. Let $ x,y \in \X. $ We say that $ y \in x^{+} $ if $ \|x+\lambda y\| \geq \|x\| $ for all $ \lambda \geq 0. $ Accordingly, we say that $ y \in x^{-} $ if $ \|x+\lambda y\| \geq \|x\| $ for all $ \lambda \leq 0. $
\end{definition}
In this context we would like to remark that while studying orthogonality of bounded linear operators, Bhattacharyya and  Grover \cite{bhatta-grover} also considered the following weaker notion of orthogonality. Let  $\X$ be a real or complex normed space and let $x,y\in \X$. We say that $x$ is {\it r-orthogonal to} $y$, denoted by $x\orbj y,$ if $ \| x+\lambda y \| \geq \| x \| $ for all $ \lambda \in \mathbb{R}. $ Of course, it is trivial to observe that in case $ \X $ is real, $x\orbj y$ if and only if $y\in x^+$ and $y\in x^-$.\\

The notion of Birkhoff-James orthogonality is intimately connected with the notion of smoothness in Banach spaces. A non-zero element $x \in \mathbb{X}$ is said to be a smooth point if there exists a unique norm one functional $f \in \mathbb{X}^{*}$ such that $f(x) = \| x \|$. We would like to note that the study of smoothness in the space of bounded linear operators between Banach spaces is an active area of interest, and we refer the readers to  \cite{PSG}, \cite{rao} and \cite{SPM}.

For  $A \in \bh,$ we use the notations $A^*$, $R(A)$, $N(A)$, to 
denote the adjoint, the range and the kernel of $A$ respectively.  If $A, B$ are self-adjoint elements of $\bh,$ we write  $A \leq B$ whenever $\langle Ax,x\rangle \leq \langle Bx,x\rangle$ 
for all $x\in \h$. An element $A\in \bh$ such that $A\geq0$ is called positive.  For every $A\geq 0$, there exists a unique positive $A^{1/2}\in \bh$ such that $A=(A^{1/2})^2$. For any $\bb\subseteq \bh$, $\bb^+$ denotes the subset of all positive operators of $\bb$.

For any $T \in \bh$, we can write
$T = Re(T) + i Im(T),$
where $Re(T) =\frac{T+T^*}{2}$ and  $Im(T)=\frac{T-T^*}{2i}$ are self-adjoint operators. This is the so called Cartesian decomposition of $T$.

Let us recall that if $M\subseteq \h$ is a closed subspace of $\h$, then $P_M$ denotes the orthogonal projection onto $M$ of $\cH$.

For any compact operator $A\in \mathcal{K}(\h)$, let $s_1(A), s_2(A),\cdots $ be the singular values of $A$, i.e., the eigenvalues of the ``absolute value-norm"  $|A| = (A^*A)^{\frac {1}{2}}$ of $A$, in decreasing order and repeated according to multiplicity. 

The notion of unitarily invariant norm (UIN)  can be defined for operators
on Hilbert spaces as a norm $||| . |||$   that satisfies  the invariance property 
$
||| UXV|||=|||X|||,
$
for any pair of unitary operators $U,V\in \bh$.   Recall that each UIN  is defined on a natural subclass $\mathcal{J}\subseteq\mathcal{K}(\h),$ called the norm ideal associated with the norm $|||.|||$.

If $A\in \mathcal{K}(\h)$ and $p>0$, let
\begin{equation} \label{defipllel}
{\|A\|}_p = \left(\sum_{i = 1}^\infty s_i(A)^p\right)^{\frac{1}{p}} 
= \left({\rm tr}|A|^p\right)^{\frac{1}{p}},
\end{equation}
where $\rm tr$ is the usual trace functional, i.e.
${\rm tr}(A)=\sum_{j=1}^{\infty} \langle Ae_j,e_j\rangle,$
and  $\{e_j\}_{j=1}^{\infty}$ is an orthonormal basis of $\mathcal{H}$. 
Equality \eqref{defipllel} defines a norm (quasi-norm) on the ideal 
$\mathbb{B}_p(\mathcal{H}) = \{A\in  \mathcal{K}(\h): {\|A\|}_p<\infty\}$ 
for $1\leq p<\infty$ ($0< p <1$), called the \textit{$p$-Schatten class}.

The study of orthogonality of bounded linear operators is also related to the following notion of operators having disjoint support. 

\begin{definition}
	Let $ \h $ be a real or complex Hilbert space. Two operators $A, B\in\bh$ have disjoint support if and only if $AB^*=B^*A=0$.
	
\end{definition}

We would like to remark that the above definition is not the original one introduced by Arazy in \cite{arazy}, but nevertheless it was proved by Lioudaki in Proposition 2.1.8 of \cite{Lioudaki} that both notions are equivalent in $\mathbb{B}_p(\mathcal{H}$). We also refer the readers to \cite{karn} for a related notion of algebraic orthogonality in the setting of $ C^{*}- $algebras.

\section{Brief outline of the paper}
The main results of this paper are demarcated into three sections. In Section 3, we exclusively study Birkhoff-James orthogonality of bounded linear operators between Hilbert spaces and Banach spaces. As mentioned in Remark $ 3.1 $ of \cite{bhatia-semrl}, the finite-dimensional Bhatia-$\breve{S}$emrl theorem can be extended to the setting of infinite-dimensional Hilbert spaces by considering norming sequences for a bounded linear operator, instead of norm attaining vectors corresponding to the said operator. However, we show that even in case of bounded linear operators between infinite-dimensional Banach spaces, it is possible to extend the the finite-dimensional Bhatia-$\breve{S}$emrl theorem verbatim, under certain additional assumptions. Let us mention here that our observations in this context can be regarded as an extension of Theorem $ 3.1 $ and Theorem $ 4.1 $ of \cite{Pawel}, where only the real case was considered. We also explore the properties of the set $ \mathcal{O}_{T, A}=\{x\in S_{\XB}: Tx\perp_{B}Ax\}, $ for any $T, A \in \bb(\XB, \Y)$ and obtain a characterization for a Hilbert space to be finite-dimensional in terms of this newly introduced notion. The study of $ \mathcal{O}_{T, A} $ may be regarded as complementary to the study of $ M_T $ done in \cite{rao}, \cite{sain_2}, \cite{sain_paul} and \cite{SPM}. In Section 4, we focus on  orthogonality of bounded linear operators and positive operators on a Hilbert space. We give a complete characterization  for isosceles orthogonality of two positive bounded linear operators.  In Section 5, we discuss some relations between the two orthogonality types, Birkhoff-James orthogonality and isosceles orthogonality. Our results in this section are valid in the context of any normed space and not just for operators between Banach spaces. We end the present paper by giving examples in the space of bounded linear operators to illustrate that Roberts orthogonality is much stronger (and therefore, restrictive) than either of Birkhoff-James orthogonality and isosceles orthogonality.

\section{Birkhoff-James Orthogonality of bounded linear operators}

We begin this section by obtaining a verbatim extension of the finite-dimensional Bhatia-$\breve{S}$emrl theorem to the infinite-dimensional setting, with an additional assumption on the norm attainment set of one of the operators. We would like to remark that such an extension was obtained by W\'{o}jcik in \cite{Pawel}, in the context of real Banach spaces, with additional geometric assumptions of strict convexity and smoothness on the range space. However, we cover the cases of both real and complex Banach spaces.

\begin{theorem}
	Let $ \XB $ and $ \Y $ be Banach spaces, either both real, or, both complex. Let $ \XB $ be reflexive. Let $ T,A \in \mathcal{K}(\XB, \Y) $ be such that $ M_T = \{ \pm x_0 \} $ in the real case and $ M_T = \{ e^{i\theta}x_0 : \theta \in [0, 2\pi) \} $ in the complex case, where $ x_0 \in S_{\XB}. $ Then $ T\obj A $ if and only if $\mathcal{O}_{T, A}\cap M_T\neq \emptyset$.
\end{theorem}

\begin{proof}
	The sufficient part of the theorem is trivially true. Let us prove only the necessary part. We will give the proof only for the complex case. The real case can be treated similarly, by applying Theorem $ 2.1 $ of \cite{SPM}. Since $ \XB $ is reflexive, $ T,A \in \mathcal{K}(\XB, \Y), $ and $ T\obj A, $ it follows from Theorem $ 2.3 $ of \cite{PSMK} that given any $ \alpha \in U = \{ \beta \in \mathbb{C} : | \beta | = 1, \arg \beta \in [0, \pi) \}, $ there exist $ x=x(\alpha),~ y=y(\alpha) \in M_T $ such that $ Ax \in (Tx)_{\alpha}^{+} = \{ z \in \mathbb{X} : \| Tx+\lambda z \| \geq \| Tx \|~ \forall \lambda = t \alpha, t \geq 0 \} $ and $ Ay \in (Ty)_{\alpha}^{-} = \{ z \in \mathbb{X} : \| Ty+\lambda z \| \geq \| Ty \|~ \forall \lambda = t \alpha, t \leq 0 \}. $ Since $ M_T = \{ e^{i\theta}x_0 : \theta \in [0, 2\pi) \}, $ we have that $ x= e^{i\theta_{1}}x_0 $ and $ y= e^{i\theta_{2}}x_0, $ for some $ \theta_1, \theta_2 \in [0, 2\pi). $ From this, using the linearity of $ T, $ it is easy to deduce that $ Tx_0 \obj Ax_0. $ In particular, it follows that $\mathcal{O}_{T, A}\cap M_T\neq \emptyset$. This completes the proof of the necessary part of the theorem and establishes it completely.
\end{proof}

Let us now study the set $ \mathcal{O}_{T, A}, $ when $ T,A \in \bb(\XB, \Y) $ are given. As an immediate application of the set $ \mathcal{O}_{T, A}, $ in the following proposition, we obtain an easy sufficient condition for Birkhoff-James orthogonality of two bounded linear operators $ T,A $ in terms of the set $ \mathcal{O}_{T, A}. $ The proof of the proposition is omitted as it is rather trivial. We would like to note that the following proposition implies that unless $ T\obj A, $ $ \mathcal{O}_{T, A} $ cannot be the whole of $ S_{\XB}. $ 

\begin{proposition}
	Let $\XB, \Y$ be any two Banach spaces, either both real, or, both complex. Let $T, A \in \bb(\XB, \Y)$. If $\mathcal{O}_{T, A}=S_{\XB}$ then $T\obj A$.
	
\end{proposition}

On the other hand, in somewhat opposite direction to the above result, we next obtain a necessary condition for Birkhoff-James orthogonality of two compact linear operators $ T, A \in  \mathcal{K}(\XB, \Y), $ in terms of the set  $\mathcal{O}_{T, A},$ when $ \XB $ is a reflexive real Banach space and $ \Y $ is any real Banach space. We would like to remark that the following theorem is motivated by Theorem $ 2.1 $ of \cite{sain_paul}, with suitable modifications. Therefore, for the sake of brevity, we make use of some of the arguments used in the proof of  Theorem $ 2.1 $ of \cite{sain_paul}.

\begin{theorem}
	Let $\XB$ be a reflexive real Banach space and $ \Y $ be any real Banach space. Let $T, A \in \mathcal{K}(\XB,\Y). $ If $T\obj A$ then $\mathcal{O}_{T, A} \neq \emptyset$.
\end{theorem}

\begin{proof}
	If possible, suppose that $\mathcal{O}_{T, A} = \emptyset.$ Therefore, it follows that given any $ x \in S_{\XB}, $ there exists $ \lambda_x \neq 0 $ such that $ \| Tx + \lambda_x Ax \| < \| Tx \|. $ Let us consider the following two sets:
	\[ V_1 = \{ z \in S_{\XB} ~:~ \|Tz+\lambda Az\| < \| Tz \|~ \mbox{for some}~ \lambda > 0 \}, \]
	\[ V_2 = \{ z \in S_{\XB} ~:~ \|Tz+\lambda Az\| < \| Tz \|~ \mbox{for some}~ \lambda < 0 \}. \]
	It is easy to check that both $ V_1 $ and $ V_2 $ are open subsets of $ S_{\XB}. $ Applying the convexity of norm, it is also easy to check that $ V_1 \cap V_2 = \emptyset. $ Moreover, it follows from our assumption of $\mathcal{O}_{T, A} = \emptyset$ that $ S_{\XB} = V_1 \cup V_2. $ Since $ S_{\XB} $ is connected, it follows that either $ V_1 = \emptyset $ or $ V_2 = \emptyset. $ Now, we will arrive at a contradiction in each of these two cases to complete the proof of the theorem. Since $T, A \in \mathcal{K}(\XB,\Y) $ and $T\obj A,$ it follows from Theorem $ 2.1 $ of \cite{SPM} that there exists $ x,y \in M_T $ such that $ Ax \in (Tx)^{+} $ and $ Ay \in (Ty)^{-}. $ We note that for any $ z \in S_{\XB}, $ $ z \in \mathcal{O}_{T, A} $ if and only if $ Az \in (Tz)^{+} $ and $ Az \in (Tz)^{-}. $ Since we have assumed that $\mathcal{O}_{T, A} = \emptyset,$ we must have, $ x \in V_2 $ and $ y \in V_1. $ This proves that $ V_1 \neq \emptyset $ and $ V_2 \neq \emptyset. $ This contradiction completes the proof of the theorem.
\end{proof}

As another application of the set $\mathcal{O}_{T, A}, $ we show that it is possible to characterize whether a given Hilbert space is finite-dimensional, using this concept.

\begin{theorem}
	A real or complex Hilbert space $\h$ is finite-dimensional if and only if for any $T, A \in \mathcal{B}(\h),$ we have, $T\obj A  \Longrightarrow \mathcal{O}_{T, A}\neq \emptyset$.
\end{theorem}

\begin{proof}
	We would like to note that the necessary part of the theorem follows directly from the Bhatia-$\breve{S}$emrl theorem. Let us prove the sufficient part. If possible, suppose that $\h$ is infinite-dimensional. It follows that there exists a countable orthonormal basis $\{e_n: n\in \N\}$ of $\h.$ Define linear operators $T$ and $A$ in $\bh$ in the following way:\\
	$Te_1=\frac12 e_1, Te_n=(1-\frac1n)e_n$ for all $n\geq 2$ and $Ae_n=\frac 1n e_n$. An easy computation reveals that the following are true:
	\[ (i)~ \|T\|=\|A\|=1,~ (ii)~ T\obj A,~ \textit{and}~ (iii)~ \mathcal{O}_{T, A}=\emptyset. \]
	
	However, this contradicts our assumption that $T\obj A  \Rightarrow \mathcal{O}_{T, A}\neq \emptyset.$ This completes the proof of the theorem.
\end{proof}

\begin{remark}
	Characterization of inner product spaces among normed spaces is a classical problem in functional analysis. We refer the readers to the excellent book \cite{amir} for more information in this regard. In recent times, in connection with the  Bhatia-$\breve{S}$emrl theorem,  Ben{\'{\i}}tez, Fern{\'a}ndez and Soriano \cite{benitez-fernandez-soriano} have obtained a characterization of finite-dimensional real Hilbert spaces among real Banach spaces. The above theorem is motivated in the same spirit, and it is valid for both real and complex Banach spaces. 
\end{remark}

From Theorem 2.2 of Sain and Paul \cite{sain_paul}, it follows that in case of $T \in \bh,$ where $ \h $ is a real or complex Hilbert space, the corresponding norm attainment set $M_T$ is either empty or it is the unit sphere of some subspace of $\h.$ Motivated by this result, it is natural to pose the following problem:\\

\emph{For which operators $ T,A \in \bh, $ it is true that $ \mathcal{O}_{T,A} $ is the unit sphere of some subspace of $ \h? $}\\

In the next proposition, we give a sufficient condition for $ T,A \in \bh $ to be such that $ \mathcal{O}_{T,A} $ is the unit sphere of some subspace of $ \h. $

\begin{proposition}
	Let $ \h $ be a real or complex Hilbert space. Let us consider the following set:
	\begin{equation*}
	\begin{split}
	\Gamma=\{(T,A) \in \bh \times \bh : \langle Tx_1, Ax_2\rangle+\langle Tx_2, Ax_1\rangle=0\\
	\text{if}~ \langle Tx_1, Ax_1\rangle=\langle Tx_2, Ax_2\rangle=0\}
	\end{split}
	\end{equation*}

	Then for any $(T,A) \in \Gamma,$ either $ \mathcal{O}_{T,A} = \emptyset $ or  $\mathcal{O}_{T, A}=S_M,$ where $M$ is a subspace of $\h$. 
\end{proposition}
\begin{proof}
	It is enough to prove that if $x_1, x_2$ satisfy $\langle Tx_1, Ax_1\rangle=\langle Tx_2, Ax_2\rangle=0$ then $\langle T(x_1+x_2), A(x_1+x_2)\rangle=0$ and $\langle T\lambda x_1, A\lambda x_1\rangle=0$ for all $\lambda \in \C$. We observe that the second condition is trivially true. On the other hand, the first condition holds true since by the hypothesis, we have,\\ 
	$\langle T(x_1+x_2), A(x_1+x_2)\rangle=(\langle Tx_1, Ax_1\rangle+\langle Tx_2, Ax_2\rangle)+(\langle Tx_1, Ax_2\rangle+\langle Tx_2, Ax_1\rangle)=0.$ This completes the proof of the proposition.
\end{proof}

\begin{remark}
	It is trivial that if $T, A \in \bh$ have disjoint support then $(T,A)\in \Gamma$. However, it is interesting to note that $ \Gamma $ contains pairs of operators that do not have disjoint support. Let $M, N$ be finite-dimensional subspaces of $\h$ such that $M\subsetneq N$. We consider $P_M$ and $P_N$ to be the orthogonal projections on $M$ and $N$  respectively. By the hypothesis, we have that $P_MP_N=P_NP_M=P_M$. Let us choose $0\neq x\in M^{\perp}\cap N$ with $\|x\|=1.$ It is easy to see that $ P_N \obj P_M, $  since  $x\in M_{P_N}$ and $\langle P_Nx, P_Mx\rangle=\langle P_Mx, x\rangle=0$. On the other hand, $0=\langle P_Ny, P_My\rangle=\langle P_My, y\rangle=\|P_My\|^2$ if and only if $y\in M^{\perp}$. Let $x_1, x_2 \in M^{\perp}$, then $\langle P_Nx_1, P_Mx_2\rangle +\langle P_Nx_2, P_Mx_1\rangle= \langle P_Mx_1, x_2\rangle +\langle P_Mx_2, x_1\rangle=0$. Therefore, we have proved that the following three statements hold true:
	\[(i)~(P_N,P_M) \in \Gamma~ (ii)~P_N\obj P_M~\textit{and}~(iii)~ P_N, P_M ~ \textit{do not have disjoint support}. \]
	
\end{remark}

\section{Orthogonality in $\bh$}
We begin this section by proving that in the context of bounded linear operators on a Hilbert space, disjoint support implies both Birkhoff-James orthogonality and isosceles orthogonality.
\begin{proposition}\label{soporte implica bj e i}
	Let $A,B\in \bh$, where $\cH$ is a real or  complex Hilbert space,  such that $B^*A=0$ , then the following holds:
	\be
	\item $A\obj B\ {\rm and}\ B\obj A$.
	\item $A\perp_{R}B$ and in particular, $A\oi B$.
	\ee
\end{proposition}
\begin{proof}
	\be
	\item Consider $h\in S_{\h}$. Then for any  $\lambda\in \mathbb{K}$ 
	$$\|(A+\lambda B)h\|^2=\|Ah\|^2+\|\lambda Bh\|^2+2|\lambda|^2{\rm Re}\left\langle B^*Ah,h \right\rangle= \|Ah\|^2+\|\lambda Bh\|^2,$$
	where $Re(z)$ denotes the usual real part of $z\in  \mathbb{K}$. Therefore, $\|A+\lambda B\|^2\geq  \|(A+\lambda B)h\|^2\geq \|Ah\|^2$ for all $h\in S_\h$, which implies that $ \|A+\lambda B\|\geq \|A\|$ for any $\lambda \in \mathbb{K}$. Interchanging the roles of $A$ and $B,$ we can obtain in a similar way that $B\obj A$.
	\item If $A,B\in\bh$ satisfy $B^*A=0,$ then for any $\lambda \in \mathbb{K},$ we have,
	\begin{equation}\label{roberts}
	\begin{split}
	\|A+\lambda B\|^2&=\sup\{\|(A+\lambda B)h\|^2:\ h\in S_\h \}\\
	&=\sup\{\|Ah\|^2+\|\lambda Bh\|^2:\ h\in S_\h \}=\|A-\lambda B\|^2.
	\end{split}
	\end{equation}
	\ee
	This completes the proof of the second part of the proposition and establishes it completely.
\end{proof}

\begin{remark}
	\be
	
	\item In particular, it follows from our previous result that for operators having disjoint support, Birkhoff-James orthogonality relation is symmetric. 
	
	\item From inequality \eqref{roberts},  we can conclude that  if  $R(A)$ and $R(B)$ are orthogonal sets then they are  Roberts orthogonal operators. We will prove  that a similar result holds when we consider any $|||\cdot|||$ UIN. 
	
	Let  $A,B\in \mathcal{J},$ where  $\mathcal{J}$ denotes the norm ideal associated with the norm,  such that $B^*A=0$. It follows that
	$|B + \lambda A| = |B - \lambda A|$ for all $\lambda\in\mathbb{K}$ and therefore it turns out that	$s_j(B+\lambda A)=s_j(B-\lambda A)$ for any $j \in \mathbb{N}$ and $|||B+\lambda A|||=|||B-\lambda A|||$.	\ee 
	
\end{remark}

However, not every pair of operators $A,B\in \bh,$ such that $A\obj B$ or $A\oi B,$ have disjoint support. This idea can be illustrated in the next example.
\begin{example} \label{ejemplo iso no sop dis}
	\be
	\item Finite dimensional case: Let $ \h $ be the two-dimensional real Hilbert space. We consider 
	Let $A=\begin{pmatrix}
	4&0\\
	0&3
	\end{pmatrix}$ and $B=\begin{pmatrix}
	0&0\\
	0&1
	\end{pmatrix}$. Then, 
	\be
	\item $\|A+B\|=\|A-B\|=4$
	and
	
	\item $A^*B=\begin{pmatrix}
	0&0\\
	0&3
	\end{pmatrix},$
	\ee 
	which implies that $ A,B $ do not have disjoint support.
	
	\item Infinite dimensional case: Let $\{e_n\}_{n\in\N}$ be an orthonormal basis for $\h$ a complex Hilbert space. Define operators $A,B:\h\to\h$ such that
	$$Ae_n=\lambda_ne_n\ \ {\text and}\ \ Be_1=0,\ Be_n=\lambda_ne_{n+1}\ \forall\ n\geq 2,$$
	with $\lambda_n\in \C$ and $|\lambda_n|=1$. Observe that $\|Ax\|=\|x\|$ for all $x\in \h$.
	
	Since $\|Ae_1\|=\|e_1\|=1=\|A\|$ and $Be_1=0$, it follows that $A\obj B$. On the other hand, $A$ and $B$ do not have disjoint support, since 
	$$A^*Be_n=A^*(\lambda_ne_{n+1})=\lambda_n(A^*e_{n+1})=\lambda_n\overline{\lambda_{n+1}}e_{n+1}\neq 0.$$
	\ee
	
\end{example}

Hereafter, unless otherwise mentioned, we consider $\cH$ to be a real Hilbert space.

\begin{theorem} \label{teo isosceles1}
	Let $A,B\in \bh$ and suppose that there exists $h_0,k_0\in \h$  such that $h_0\in M_{A+B} $ and
	$k_0\in M_{A-B} $. Then the following assertions are true.
	\be
	\item If $\left\langle Ah_0,Bh_0 \right\rangle\leq 0$ and $\left\langle Ak_0,Bk_0 \right\rangle\geq 0$,
	then $A\oi B$.
	\item If $A\oi B$ then $\left\langle Ah_0,Bh_0 \right\rangle\geq 0$ and $\left\langle Ak_0,Bk_0 \right\rangle\leq 0$.
	\ee
\end{theorem}

\begin{proof}
	\be
	\item Assume that all the conditions of the statement are satisfied. Let $f,g:\h\to\R$ be given by
	\begin{eqnarray}\label{f y g}
	f(h)&=&\|(A+B)h\|^2=\|Ah\|^2+\|Bh\|^2+2\left\langle Ah,Bh \right\rangle\ {\rm and}\nonumber\\
	g(h)&=&\|(A-B)h\|^2=\|Ah\|^2+\|Bh\|^2-2 \left\langle Ah,Bh \right\rangle.\
	\end{eqnarray}
	Then,  $f(h)-g(h)=4  \left\langle Ah,Bh \right\rangle\ $. Suppose that 
	$$g(k_0)=\|A-B\|^2<\|A+B\|^2=f(h_0)\Rightarrow g(h_0)\leq g(k_0)<f(h_0).$$
	Thus, $0<f(h_0)-g(h_0)=4\left\langle Ah_0,Bh_0 \right\rangle$, which is a contradiction. Hence, $g(k_0)\geq f(h_0)$. Analogously, it can be proved that $f(h_0)\geq g(k_0)$. Finally, 
	$$\|A+B\|^2=f(h_0)=g(k_0)=\|A-B\|^2,$$
	which implies $A\oi B$.
	\item We only prove the first inequality, the other can be obtained with a similar argument. By the real polarization formula we get
	$$
	\left\langle Ah_0,Bh_0 \right\rangle=\frac14 [\|(A+B)h_0\|^2-\|(A-B)h_0 \| ^2]\geq\frac14 [ \|A+B\|^2-\|A-B\|^2]=0.
	$$
	\ee
\end{proof}
\begin{remark}
	Suppose that in Theorem \ref{teo isosceles1},  $h_0$ and $k_0$ also satisfy
	$$\left\langle Ah_0,Bh_0 \right\rangle=\left\langle Ak_0,Bk_0 \right\rangle=0.$$
	Then, there exists $h_1\in S_\h$ such that
	$\|(A+B)h_1\|=\|(A-B)h_1\|$. 
	
	It can be easily proved using polarization formula and hypothesis,
	\begin{equation*}
	\begin{split}
	0&=\left\langle Ah_0,Bh_0 \right\rangle\ =\frac{1}{4}\left[\|(A+B)h_0\|^2-\|(A-B)h_0\|^2 \right] 
	\end{split}
	\end{equation*}
	and this implies
	$\|(A-B)h_0\|=\|(A+B)h_0\|=\|(A-B)k_0\|$, where last equality is due to isosceles orthogonality between $A$ and $B$ previously proved. By a similar argument, it can be proved that
	$\|(A+B)k_0\|=\|(A+B)h_0\|.$
	The proof is completed by taking $h_1\in \{h_0;k_0\}$.
	
\end{remark}
The following result combines  Theorem \ref{teo isosceles1} and last remark.  
\begin{corollary} \label{isosceles implica desigualdad}
	Let $A,B\in \bh$ and suppose that there exists $h_1\in M_{A+B}\cap M_{A-B}$ such that
	$\left\langle Ah_1,Bh_1 \right\rangle=0$. Then $A\oi B$,
	$$\|A\|^2+\|B\|^2\leq \|A+B\|^2+\|A-B\|^2\leq 2\left( \|A\|^2+\|B\|^2\right),$$
	and if $h_1\notin N(A)\cup N(B)$ then $\|A+B\|^2=\|A-B\|^2=2$.
\end{corollary}

\begin{proof}
	It was proved in Theorem \ref{teo isosceles1} that, under these hypothesis, $A\oi B$. Moreover,
	\begin{equation*}
	\begin{split}
	\|A+B\|^2&=\|(A+B)h_1\|^2=\|Ah_1\|^2+\|Bh_1\|^2\leq \|A\|^2+\|B\|^2\ {\rm and}\\
	\|A-B\|^2&=\|(A-B)h_1\|^2=\|Ah_1\|^2+\|Bh_1\|^2\leq \|A\|^2+\|B\|^2.
	\end{split}
	\end{equation*}
	Then,
	$$\|A+B\|^2+\|A-B\|^2\leq 2\left( \|A\|^2+\|B\|^2\right).$$
	On the other hand,
	$$\|A\|^2+\|B\|^2\leq 2\max(\|A\|^2;\|B\|^2)\leq \|A+B\|^2+\|A-B\|^2,$$
	
	Finally, it is easy to see that
	$$
	\|(A+B)h_1\|^2\|Bh_1\|^2-\|Ah_1\|^2\|Bh_1\|^2=|\langle (A+B)h_1, Bh_1\rangle|^2-|\langle Ah_1, Bh_1\rangle|^2.
	$$
	If we assume that $h_1\notin N(A)\cup N(B)$, then $\|A+B\|^2=1+\|Ah_1\|^2$. By symmetry we obtain that $\|A+B\|^2=1+\|Bh_1\|^2$.
	
	Now, by the Parallelogram law we get 
	$$
	\|A+B\|^2+\|A-B\|^2=\|(A+B)h_1\|^2+\|(A-B)h_1\|^2=2(\|Ah_1\|^2+\|Bh_1\|^2).
	$$
	It follows that $\|Bh_1\|=1$ and $\|A+B\|^2=2$.
\end{proof}

The following result is other  characterization of isosceles orthogonality of bounded linear operators in finite-dimensional real Hilbert spaces, with an additional condition.

\begin{theorem} \label{teo equiv iso}
	Let $\h$ be a finite-dimensional real Hilbert space  with $\dim(\h)=n$ and $A,B\in \bh$. Suppose that $M_{A+B}=S_{\h_1}$,  $M_{A-B}=S_{\h_2}$, and $\dim(\h_1)+\dim(\h_2) > n$. Then, the following statements are equivalent:
	\be
	\item $A\oi B$.
	\item There exists $x_0\in M_{A+B}\cap M_{A-B}$ such that $\left\langle Ax_0,Bx_0 \right\rangle=0$
	\ee
\end{theorem}
\begin{proof}
	
	Suppose that statement $(1)$ holds. Since $\dim(\h_1)+\dim(\h_2) > n$, there exists $0\neq x_1\in \h_1\cap \h_2$. Consider $x_0= \frac{x_1}{\|x_1\|}$, then $x_0\in M_{A+B}\cap M_{A-B}$. Also, by hypothesis and real polarization formula
	$$\left\langle Ax_0,Bx_0 \right\rangle=\frac 1 4 \left(\|(A+B)x_0\|^2-\|(A-B)x_0\|^2 \right) =\frac 1 4 \left(\|A+B\|^2-\|A-B\|^2 \right)=0.$$
	Conversely, the other implication is a consequence of Corollary \ref{isosceles implica desigualdad}.
\end{proof}
\begin{remark}
	If $\h$ is a complex Hilbert space, isosceles orthogonality must be defined as in \eqref{defiiso2} and the previous statements  can be generalized to this context  with  proofs which are essentially the same as the real case.  
	For example, in case of a complex Hilbert space $ \h, $ Theorem \ref{teo equiv iso} can be stated in the following way:
	
	Let $A,B\in \bh$  with $\dim(\h)=n$. Suppose that there exist $\mathcal{W}_1, \mathcal{W}_2, \h_1, \h_2$ subspaces of $\h$ such that
	$$
	\left\lbrace \begin{array}{llll}
	M_{A+B}=S_{\h_1},&  M_{A-B}=S_{\h_2}& \text{and}& \dim(\h_1)+\dim(\h_2) > n,\\
	M_{A+iB}=S_{\mathcal{W}_1},&  M_{A-iB}=S_{\mathcal{W}_2}& \text{and}& \dim(\mathcal{W}_1)+\dim(\mathcal{W}_2) > n.\\
	\end{array}\right. $$
	Then, the following statement are equivalent:
	\be
	\item $A\oi B$.
	\item There exist
	$
	\left\lbrace \begin{array}{llll}
	h_0\in M_{A+B}\cap M_{A-B}&\text{such that}&\text{Re}\left\langle Ah_0,Bh_0 \right\rangle=0,\\
	k_0\in M_{A+iB}\cap M_{A-iB}&\text{such that}&\text{Re}\left\langle Ak_0,Bk_0 \right\rangle=0.\\
	\end{array}\right. $
	\ee
\end{remark}
In the cone of positive
operators between Hilbert spaces, as in a real normed space,  we use isosceles orthogonality notion as in \eqref{defiiso}.
Kittaneh proved in \cite{kittaneh_laa_2004},  that
if $A,B\in\bh^+$, then
\begin{equation}\label{cota kittaneh iso}
\begin{split}
\max(\|A\|,  \|B\|)-\|A^{1/2}B^{1/2}\|&\leq \|A-B\|\\
&\leq \max(\|A\|, \|B\|)\leq \|A+B\|\\
&\leq \max(\|A\|, \|B\|)+\|A^{1/2}B^{1/2}\|.
\end{split}
\end{equation}
The above  inequalities  are useful in the study of isosceles orthogonality in $\bh^+$.  As an immediate consequence, we deduce that if $A,B$ are positve operators  and $A^{1/2}B^{1/2}=0$, then $A\oi B$.

\begin{proposition} \label{iso max}
	Let $A,B\in \bh^+$. Then, the following conditions are equivalent:
	\be
	\item $A\oi B$.
	\item $\|A+B\|=\|A-B\|=\max(\|A\|, \|B\|)$.
	\ee
\end{proposition}
\begin{proof}
	$(1)\Rightarrow (2)$ Suppose that $A\oi B$. Equation \eqref{cota kittaneh iso} states 
	$$\|A-B\|\leq \max(\|A\|, \|B\|)\leq \|A+B\|$$
	for any $A,B\in\bh^+$. The desired result is now immediate.
	
	The converse implication is trivial.
\end{proof}

In order to simplify the exposition, we introduce the following notations. Given $A, B \in \bh$ we define
$$
M(A, B)=\left\lbrace \begin{array}{llll}
A &\text{if}& \|B\|\leq \|A\|\\
B &\text{if}& \|A\|< \|B\|\\
\end{array}\right. $$
and 
$$
m(A, B)=\left\lbrace \begin{array}{llll}
B &\text{if}& \|B\|\leq \|A\|\\
A &\text{if}& \|A\|< \|B\|\\
\end{array}\right. .$$

In the next statement we obtain a characterization for isosceles orthogonality when $A$ and $B$ are positive operators.   

\begin{theorem}\label{cariso}
	Let $A,B\in \bh^+$, then  $A\oi B$ if and only if there exists a sequence $\{x_n\} \subset S_{\h}$  such that $\lim\limits_{n\to \infty}\|M(A, B)x_n\|=\|A+B\|$ and $\lim\limits_{n\to \infty}Re\langle BA x_n, x_n\rangle\leq0.$
\end{theorem}

\begin{proof}
	Let $A\oi B$. By Proposition \ref{iso max} we have $\|A+B\|=\|A-B\|=\max(\|A\|, \|B\|)$. 
	Let $\{x_n\}$ a sequnce of unit vectors in $\h$ such that $\lim\limits_{n\to \infty}\|M(A, B)x_n\|=\|A+B\|$. Since $\{(A+B)x_n\} $ is a bounded sequence, it has a convergent sequence and  without loss of generality we assume that 
	\begin{equation}
	\lim\limits_{n\to \infty}\|(A+B)x_n\|^2\leq \|A+B\|^2=\lim\limits_{n\to \infty}\|M(A, B)x_n\|^2\nonumber.
	\end{equation}
	Thus 	$\lim\limits_{n\to \infty}\|m(A, B)x_n\|^2+2Re\langle BA x_n, x_n\rangle\leq 0$ and $\lim\limits_{n\to \infty} Re\langle BA x_n, x_n\rangle\leq 0$.

	Conversely, by the hypothesis  there exists a sequence $\{x_n\} \subset S_{\h}$  such that $\|A+B\|=\lim\limits_{n\to \infty}\|M(A, B)x_n\|$ and  $\lim\limits_{n\to \infty}Re\langle BA x_n, x_n\rangle\leq0.$  From the first condition we have that $\|A+B\|=\max(\|A\|, \|B\|)$.
	Suppose that $A$ and $B$ are not isosceles orthogonal, this means that $\|A-B\|<\|A+B\|$.
	Hence, we have
	\begin{equation*}
	\lim\limits_{n\to \infty}\|(A-B)x_n\|^2\leq\|A-B\|^2 <\|A+B\|^2=\lim\limits_{n\to \infty}\|M(A, B)x_n\|^2.
	\end{equation*}
	This implies that
	$0\leq \lim\limits_{n\to \infty}\|(A+B)x_n\|^2< \lim\limits_{n\to \infty}2 Re\langle BA x_n, x_n\rangle$,
	which is a contradiction. Therefore,  $A\oi B$.
\end{proof}

As a consequence of the previous statement, we have the following characterization of the  isosceles orthogonal condition for elements of $\bh^+$ such that $Re(BA)\geq 0$. This extra condition is well known and it is related with acreetive operator theory. We recall that an operator $T\in \bh$ is called acreetive if in  its Cartesian decomposition
$Re(T)$ is positive.  

\begin{corollary}
	Let $A,B\in \bh^+$ and suppose that $BA$ is acreetive,  then  $A\oi B$ if and only if there exists a sequence $\{x_n\} \subset S_{\h}$  such that $\lim\limits_{n\to \infty}\|(A+B)x_n\|=\|A+B\|$ and $\lim\limits_{n\to \infty}\langle BA x_n, x_n\rangle=0.$
\end{corollary}
\begin{proof}
	Let $A\oi B$. ByTheorem \ref{cariso},  there exists a sequence of unit vectors in $\h$, $\{x_n\}$,  such that $\lim\limits_{n\to \infty}\|M(A, B)x_n\|=\|A+B\|$ and $\lim\limits_{n\to \infty}Re\langle BA x_n, x_n\rangle\leq0.$ It follows from 
	\begin{eqnarray}
	\|M(A, B)x_n\|^2 \nonumber &\leq& \|M(A, B)x_n\|^2+\|m(A, B)x_n\|^2+2Re\langle BA x_n, x_n\rangle\\& =&\|(A+B)x_n\|^2 \leq \|A+B\|^2\nonumber,
	\end{eqnarray}
	that $\lim\limits_{n\to \infty}\|(A+B)x_n\|=\|A+B\|$,  $\lim\limits_{n\to \infty}\|m(A, B)x_n\|=0$ and $\lim\limits_{n\to \infty}Re\langle BA x_n, x_n\rangle=0.$
	
	Since $|\langle BA x_n, x_n\rangle| \leq \|M(A, B)x_n\| \|m(A, B) x_n\|$ we infer that $\lim\limits_{n\to \infty}\langle BA x_n, x_n\rangle=0.$
	
	Conversely, by the hypothesis  there exists a sequence $\{x_n\} \subset S_{\h}$  such that $\lim\limits_{n\to \infty}\|(A+B)x_n\|=\|A+B\|$ and  $\lim\limits_{n\to \infty}\langle BA x_n, x_n\rangle=0.$ Suppose that $A$ and $B$ are not isosceles orthogonal, then $\|A-B\|<\|A+B\|$.
	Hence we have
	\begin{equation*}
	\lim\limits_{n\to \infty}\|(A-B)x_n\|^2\leq\|A-B\|^2 <\|A+B\|^2=\lim\limits_{n\to \infty}\|(A+B)x_n\|^2.
	\end{equation*}
	This implies that
	$0<4\lim\limits_{n\to \infty}Re\langle BA x_n, x_n\rangle$,
	which is contradiction. Therefore,  $A\oi B$.
\end{proof}

\begin{corollary}
	Let $A,B\in \bh^+$ and suppose that $BA$ is acreetive. If $A\oi B$ then $A\obj B$ or $B\obj A$.
\end{corollary}

\begin{remark}
	Let $A,B\in \bh$ such that $A,B>0$ (i.e. positive and invertible). We will prove that $A$ can not be isosceles orthogonal to $B$. Indeed, suppose that $A\oi B$. Using formulas (26) and (27) in \cite{kittaneh_laa_2004} we have
	$$\max(\|A\|, \|B\|)+\min(\|A^{-1}\|^{-1}, \|B^{-1}\|^{-1})\leq \|A+B\|$$
	$$=\|A-B\|\leq \max(\|A\|, \|B\|)-\min(\|A^{-1}\|^{-1},\|B^{-1}\|^{-1}), $$
	$$\Rightarrow 0<2\min(\|A^{-1}\|^{-1}, \|B^{-1}\|^{-1})\leq 0,$$ 
	which is a contradiction.  
	
	A natural question is whether $ A $ and $ B $ can be Birkhoff-James orthogonal. Suppose that $A\obj B$, then for any $\lambda >0,$ we have,
	$$\|A\|\leq \|A- \lambda B\|\leq \max(\|A\|, \lambda \|B\|)-\min(\|A^{-1}\|^{-1}, \lambda \|B^{-1}\|^{-1})<\max(\|A\|, \lambda \|B\|),$$which is clearly a contradiction. This proves that $A$ is not Birkhoff-James orthogonal to $B$.
\end{remark}
Next results and comments are related to isosceles orthogonality between positive operators and projections.

\begin{proposition}
	For any $P\in \bh^+$,  
	$P\oi I$ if and only if  $P=0.$
\end{proposition}
\begin{proof}
	By Proposition \ref{iso max},  $\|P-I\|=\|P+I\|=\max\{1;\|P\|\}$. 
	Suppose $\|P\|\geq 1.$ Then $\|P-I\|=\|P+I\|=\|P\|$, which is a contradiction since $\|(P+I)x\|^2 >\|Px\|^2$ for all $x\neq 0$. On the other hand, if $\|P\|<1$ and $P\neq 0$, consider $x\neq 0$, $\|x\|=1$ and $Px\neq 0$. Then, 
	$$0\leq \|Px+x\|^2=\|Px\|^2+2\left\langle Px, x\right\rangle+\|x\|^2\leq 1\Rightarrow \|Px\|^2+2\left\langle Px, x\right\rangle=0$$
	$$\Rightarrow\|Px\|^2=2\left\langle Px, x\right\rangle=0. $$
	This completes the proof.	
\end{proof}
From the above, in finite-dimensional context, we deduce if $A=UP\neq 0$, with $U$ unitary and $P\geq 0$, then $A$ and $U$ can not be isosceles orthogonal. Moreover, using Theorem IX.7.2 in \cite{bhatia_book}, we obtain that $\|A-U\|<\|A+U\|$.

In case of orthogonal projections isosceles orthogonality implies disjoint support, as we show in the next result.
\begin{proposition} \label{ortogonalidad proyecciones}
	Let $P_S,P_T$ be orthogonal projections with $S\neq T$. Then, $P_S\oi P_T$ if and only if $ P_SP_T=0$.  In particular, if $P_S\oi P_T$ then $P_S\obj P_T$ and $P_T\obj P_S$.
\end{proposition}
\begin{proof}
	Suppose $P_S\oi P_T$. By equation \eqref{cota kittaneh iso}, 
	$$\|P_T-P_S\|\leq \max(\|P_S\|,\|P_T\|)\leq \|P_S+P_T\|.$$
	Then by hypothesis, $\|P_S-P_T\|=\|P_S+P_T\|=1$.	
	On the other hand, by \cite{kato}, $\|P_S-P_T\|=1$ if and only if  $P_T$ and $P_S$ commute and, by \cite{deutsch}, we have $P_TP_S=P_SP_T=P_{T\cap S}$. If there exists  $h\in T\cap S $ with $h\neq 0$, then $\|(P_T+P_S)h\|=2\|h\|$ and this implies $\|P_T+P_S\|\geq 2$, which is a contradiction. Therefore, $T\cap S=\{0\}$ and $P_SP_T=P_{\{0\}}=0$.
	
\end{proof}
However, it is not true that there exists a equivalence between disjoint support and Birkhoff-James orthogonality, even in the case of orthogonal projections. For instance, consider in $\R^3$ the projections onto the planes $z=0$ and $x=0$, $P_{z=0}$ and $P_{x=0}$, respectively. Clearly, $P_{z=0}\obj P_{x=0}$ but $P_{z=0}\cap P_{x=0}=P_{x=0,z=0}\neq 0$, which means they have not disjoint support.

\section{Relations between different types of orthogonality}
In this short section we study the relations between Birkhoff-James orthogonality and isosceles orthogonality. Before proceeding any further, let us mention the following fact that serves as a motivation behind our exploration in this section.

Bottazzi et. al. studied the equivalence of Birkhoff-James orthogonality and isosceles orthogonality of positive operators $A,B$ in a $p-$Schatten ideal in \cite{bcmz}. Indeed, for every $A, B \in \bh^+$ and  $1<p\leq 2$, they proved that
$A\objp B$ and $A\oip B$ are equivalent notions.

However, it is not difficult to observe that there are many examples in $\bh$ (and more generally, in Banach spaces) that show $\obj$ and $\oi$ are independent orthogonality types and none of them imply the other. Our purpose in this section is to establish relations between these two orthogonality types, in the sense that we determine which additional conditions may be required to have $``\obj \Rightarrow \oi"$ and vice versa. Recall that $\X$ is a real or complex normed space.
\begin{proposition}
	Let $x,y\in \X$ and assume that $(x+y)\obj y$ and $(x-y)\obj y.$ Then $x\oi y$.
\end{proposition}

\begin{proof}
	By the hypothesis, we have,
	$\|x+y\|\leq \|x+y+\lambda y\|\ \forall\ \lambda \in \mathbb{K}.$
	Taking $\beta=1+\lambda, $ we have, $\|x+y\|\leq \|x+\beta y\|$. In particular for $\beta=-1$, we get $\|x+y\|\leq \|x-y\|$. Analogously, from the hypothesis $(x-y)\obj y,$  we obtain $\|x-y\|\leq \|x+y\|$. This proves that $x\oi y$ and completes the proof of the proposition.
\end{proof}

In order to address the converse question, we introduce the concept of strongly isosceles orthogonality in real Banach spaces.
\begin{definition}
	Let $x,y\in \X$. We say that $x$ is strongly isosceles orthogonal to $y$, written as $x\osi y$ if
	\begin{enumerate}
		\item $x\oi y$.
		\item there exists a real sequence $\{\lambda_n\}_{n\in \N}$, with $\lambda_n>0$, such that $\lim\limits_{n\to\infty}\lambda_n=0$ and $x\oi\lambda_n y$ for all  $n\in \N$.
	\end{enumerate}
\end{definition}

In view of the above definition, we obtain the following statement.

\begin{theorem}
	Let $x,y\in \X$. Then $x\osi y$ implies  $x\orbj y$ and in particular if $\X$ is a real normed space then $x\osi y$ implies $x\obj y$.
\end{theorem}
\begin{proof}
	By Theorem 4.1 of \cite{james2}, it follows that $ \|x\|\leq \|x+\lambda y\|$ for $\lambda\in  \R$, $|\lambda|\geq 1.$ Therefore, we only have to prove the Birkhoff-James orthogonality condition for $|\lambda|<1$. Let $\lambda_0\in \R$ such that $|\lambda_0|<1$. By the hypothesis, there exists $n_0\in \N$ such that $0<\lambda_{n_0}<|\lambda_0|$ and $x\oi \lambda_{n_0}y$. Then, $\beta=\frac{\lambda_0}{\lambda_{n_0}}$ satisfies $|\beta|>1$ and, by the cited result of James \cite{james2}, we get
	$\|x\|\leq \|x+\beta\lambda_{n_0} y\|$, since $|\beta|>1$ and $x\oi \lambda_{n_0}y$. However, this is clearly equivalent to the following: 
	$$\|x\|\leq \left\| x+\frac{\lambda_0}{\lambda_{n_0}}\lambda_{n_0}y\right\| =\|x+\lambda_0 y\|.$$
	This completes the proof of the proposition.
\end{proof}

\begin{remark}
	Using the convexity of the norm function, it is possible to show that only condition  (ii) in the definition of strongly isosceles orthogonality is sufficient to ensure Birkhoff-James orthogonality of the corresponding elements. However, we include condition (i) in the definition of strongly isosceles orthogonality because we are trying to address the question that asks Isosceles orthogonality, along with which additional conditions, implies Birkhoff-James orthogonality.
\end{remark}

We have already discussed that Roberts orthogonality is stronger and more restrictive than either of Birkhoff-James orthogonality and isosceles orthogonality. Moreover, it is obvious that $A\oro B\Rightarrow A\osi B.$ In the next two examples we show that the converse of this statement is not necessarily true. We deliberately give the examples using different norms on $ \bh, $ to make them more illustrative.

\begin{example} \label{si not roberts in bh}
	Let $ \h $ be the two-dimensional real Hilbert space. Consider the Banach space $\bh,$ endowed with the usual uniform norm. Let $A$ and $B$ matrices considered in Example \ref{ejemplo iso no sop dis} item (1). Then,
	\be
	\item $\|A+B\|=\|A-B\|=4$
	\item Let $\{\lambda_n\}_{n\in \N}$ be a sequence such that $\lambda_n\in (0,1)$ and $\lambda_n\to 0.$ We have,
	$$A+\lambda_nB=\begin{pmatrix}
	4&0\\
	0&3+\lambda_n
	\end{pmatrix}\Rightarrow \|A+\lambda_nB\|=4.$$
	On the other hand,
	$$A-\lambda_nB=\begin{pmatrix}
	4&0\\
	0&3-\lambda_n
	\end{pmatrix}\Rightarrow \|A-\lambda_nB\|=4.$$
	Conditions (1) and (2) together imply that $A\osi B$.
	\item However, if we consider $\lambda=5$, we have 
	$$A+5B=\begin{pmatrix}
	4&0\\
	0&8
	\end{pmatrix}\Rightarrow \|A+5B\|=8\ \ \text{and}\ A-5B=\begin{pmatrix}
	4&0\\
	0&-2
	\end{pmatrix}\Rightarrow \|A-5B\|=4.$$
	Therefore, $A$ is not Roberts orthogonal to $B$. 
	\ee
\end{example}
\begin{example} \label{si not roberts in b1}
	Let $ \h $ be the two-dimensional real Hilbert space. Consider the Banach space $\bh,$ but now endowed with the $1-$Schatten norm.
	
	Let $A=\begin{pmatrix}
	1&0\\
	0&-2
	\end{pmatrix}$ and $I=\begin{pmatrix}
	1&0\\
	0&1
	\end{pmatrix}\in\bh$. Then,
	\be
	\item $\|A+I\|_1=2+1=\|A-I\|_1$.
	\item Let $\{\lambda_n\}_{n\in \N}$ be a sequence such that $\lambda_n\in (0,1)$ and $\lambda_n\to 0$,
	$$A+\lambda_nI=\begin{pmatrix}
	1+\lambda_n&0\\
	0&-2+\lambda_n
	\end{pmatrix}\Rightarrow \|A+\lambda_nB\|_1=|1+\lambda_n|+|-2+\lambda_n|=3$$
	$$A-\lambda_nI=\begin{pmatrix}
	1-\lambda_n&0\\
	0&-2-\lambda_n
	\end{pmatrix}\Rightarrow \|A-\lambda_nB\|_1=|1-\lambda_n|+|-2-\lambda_n|=3.$$
	Clearly, conditions (1) and (2) together imply that $A\osi I$.
	\item But if we consider $\lambda=2$, we have, 
	$$A+2I=\begin{pmatrix}
	3&0\\
	0&0
	\end{pmatrix}\Rightarrow \|A+2I\|_1=3\ \ \text{and}\ A-2I=\begin{pmatrix}
	-1&0\\
	0&-4
	\end{pmatrix}\Rightarrow \|A-2I\|_1=5.$$
	This proves that $A$ is not Roberts orthogonal to $B$. 
	\ee	
\end{example}

\medskip

\section{Acknowledgements}
	The research of Dr. Tamara Bottazzi and Dr. Cristian Conde is partially supported by National Scientific and Technical Resesarch Conuncil of Argentina (CONICET)and ANPCyT PICT 2017-2522. 
	
	The research of Dr. Debmalya Sain is sponsored by Dr. D. S. Kothari Post-doctoral Fellowship, under the mentorship of Professor Gadadhar Misra. Dr. Sain feels elated to acknowledge the wonderful hospitality of the lovely couple Mr. Subhro Jana and Mrs. Poulomi Mallik.


%
%



\end{document}